\documentclass{amsart}
\usepackage{amsfonts}
\usepackage{amsfonts,amssymb,amsmath,amsthm}
\usepackage{url}
\usepackage{enumerate}\usepackage{geometry}
\usepackage{graphicx}\usepackage{float}

\newtheorem{thm}{Theorem}[section]
\newtheorem{cor}[thm]{Corollary}

\newtheorem{lem}[thm]{Lemma}

\newtheorem{rem}[thm]{Remark}

\newcommand{\be}{\begin{equation}}
\newcommand{\ee}{\end{equation}}
\newcommand{\ben}{\begin{enumerate}}
\newcommand{\een}{\end{enumerate}}
\newcommand{\beq}{\begin{eqnarray}}
\newcommand{\eeq}{\end{eqnarray}}
\newcommand{\beqn}{\begin{eqnarray*}}
\newcommand{\eeqn}{\end{eqnarray*}}

\newcommand{\w}{{\wedge}}

\newcommand{\la}{{\langle}}
\newcommand{\ra}{{\rangle}}

\begin{document}

\title[simons' type formula and contact angle]{\bf  A simons' type formula for cmc surfaces in homogeneous $3$-manifolds}

\author{Ningwei Cui}
 \footnotetext{This work is partially supported by NSFC (No. 11401490).}
\date{}

\address{address:
School of Mathematics, Southwest Jiaotong
 University, Chengdu, 610031, P.R.China.}
 \email{ningweicui@gmail.com}

\maketitle
\begin{abstract}
{In this paper, we give a Simons' type formula for the cmc surfaces in homogeneous $3$-manifolds $E(\kappa,\tau)$, $\tau\neq0$. As an application, we give a rigidity result in the case of  $\kappa> 4\tau^2$ for the cmc surfaces under a pinching assumption of the second fundamental form.}
\end{abstract}

\section{Introduction}

It is well known that the simply connected 3-dimensional homogeneous Riemannian manifolds have the isometry groups of dimension $3$, $4$ and $6$. When the dimension of the isometry group is $6$, then we have a space form.  When the dimension of the isometry group is $3$, then the manifold has the geometry of the Lie group $Sol_3$. 
 We denote $E(\kappa,\tau)$, $\kappa\neq 4\tau^2$, as the homogeneous $3$-manifolds whose isometry groups are of dimension $4$, which are fibrations over  $2$-dimensional simply connected space forms  $\Bbb{M}^2(\kappa)$ of constant curvature $\kappa$.  In other words, there exists a Riemannian submersion $\Pi:E(\kappa,\tau)\rightarrow\Bbb{M}^2(\kappa)$, where the constant number $\tau$ is the bundle curvature. The fibers are geodesics and there exists a one-parameter family of translations along the fibers, generated by a unit Killing vector field $\xi$.
When $\tau=0$, we get a product manifold $\Bbb{M}^2(\kappa)\times \Bbb{R}$. When $\tau\neq0$, the manifolds are of three types:  the Berger sphere  $\Bbb{S}_{\kappa,\tau}^3$ ($\kappa>0$),  the Heisenberg Group $Nil_3$ ($\kappa=0$) and  $P\widetilde{SL_2}$
($\kappa<0$).  The so called {\em Hopf cylinder} is defined as the preimage $\Pi^{-1}(\gamma)$ of a regular closed curve $\gamma$ in $\Bbb{M}^2(\kappa)$. When considering $\Bbb{S}_{\kappa,\tau}^3$ and $P\widetilde{SL_2}$ (the fibers are circles), the corresponding Hopf cylinder is also called the {\em Hopf torus}. For more details on $E(\kappa,\tau)$, we refer to \cite{Da}.
The surface with constant mean curvature $H$ is called the cmc $H$-surface. There are a lot of researches on the geometry of  $E(\kappa,\tau)$, for instance (\cite{AbRosenb}, \cite{AbRosenb2}, \cite{BaFe}, \cite{brs}, [9-11],  \cite{Hu}, [17-21]), among others.   A study on the cmc  $H$-surfaces  in more general spaces can be found in \cite{MePe}.

  In this paper we focus on $E(\kappa,\tau)$, $\tau\neq 0$, and we get a Simons' type formula for cmc $H$-surfaces  which can be  stated as follows. Notice that the method used in this paper is elementary.
%t is well known that the tangent vector field of the orbits of the isometry group acting  on $E(\kappa,\tau)$  is a  Killing vector field $\xi$ of unit length. 

\begin{thm}\label{mainthm}
Let $f:M\rightarrow E(\kappa,\tau)$ ($\tau\neq 0$) be an immersion of a compact cmc $H$-surface. Denote $\Phi:=A-HI$ and the angle function by $C:=\la \xi,N\ra$, where $N$ is a unit normal vector field on $M$ and $A$ is the second fundamental form. Then it holds the Simons' type formula
\be \int_{M}\Big\{|\Phi|^4-[2(H^2+\tau^2)+(\kappa-4\tau^2)(5C^2-1)]|\Phi|^2+2(\kappa-4\tau^2)(H^2+\tau^2)(3C^2-1)\Big\}d\sigma \geq0, \label{Simon}\ee
where the equality holds if and only of $f$ is of parallel second fundamental form, and in this case $M$ is a Hopf cylinder with $|A|^2=2(2H^2+\tau^2)$.
\end{thm}

\begin{rem}  It is well known that the compact minimal immersed surface $M$ in the unit sphere $\Bbb{S}^3$ satisfies the Simons'  integral inequality \[\int_M|A|^2(|A|^2-2)d\sigma\geq 0,\]
where the equality holds if and only of $M$ is the great $2$-sphere or the Clifford torus (\cite{CCK}).
\end{rem}

\begin{rem} For the minimal surface in $E(\kappa,\tau)$ ($\tau\neq 0$), $H=0$,  the corresponding Simons' type formula was recently given in \cite{Hu}. \end{rem}

\section{Preliminaries}

For an immersed surface $M$ in $E(\kappa,\tau)$, define the {\it contact angle} $\beta(p)$ at $p\in M$ by the angle between the vector $\xi(p)$ and the tangent space $T_pM$. By definition, the contact function $C=\sin\beta$. Since  the distribution $\la\xi\ra^\perp$ in $E(\kappa,\tau)$ is not integrable (see \cite{Aebischer}), we consider the set $\mathcal{W}=\{p\in M: \beta(p)\neq \pm\frac{\pi}{2}\}\subset M$,  whose complementary $\mathcal{W}^{C}=\{p\in M: \beta(p)= \pm\frac{\pi}{2}\}$  has empty interior  by Frobenius Theorem.  In this paper we will work  by making use of the method of moving frames on $\mathcal{W}$ and then extend the formulas to the whole surface by the continuity.

\subsection{Adapted frame} 
In this section we introduce the method of choosing an appropriate  orthonomal  frame $\{e_1,e_2,e_3\}$ on $\mathcal{W}\subset M$ to study its geometry.  It is well known that $E(\kappa,\tau)$  admits an orthonormal frame $\{f_1,f_2,f_3\}$ with $f_3=\xi$, satisfying (see \cite{Da}) \[[f_1,f_2]=-2\tau f_3,\ \ \ [f_2,f_3]=-\kappa/(2\tau) f_1\ \ \ \mbox{and}\ \ \ [f_3,f_1]=-\kappa/(2\tau)  f_2.\label{framelee}\]
 It is straightforward to verify that the coframe $\{w^1,w^2,w^3\}$ of $\{f_1,f_2,f_3\}$ satisfies 
\[ dw^1=\kappa/(2\tau)  w^2\w w^3,\ \ \ dw^2=\kappa/(2\tau) w^3\w w^1\ \ \ \mbox{and}\ \ \ dw^3=2\tau w^1\w w^2,\label{BBBB}\] and the Levi-Civita connection $1$-forms are given by
\[ w^1_2=(\kappa/\tau-\tau)w^3,\ \ \ w^1_3=-\tau w^2\ \ \ \mbox{and}\ \ \ w^2_3=\tau w^1.\label{connection}\]
An interesting observation (see Lemma 1 in \cite{CuiGomes}) is  that,  if one rotates $f_1$ and  $f_2$ but leaves $f_3$ fixed,  the relation of $w^1_3$ and  $w^2_3$ will not change ($w^1_2$ may change). Thus,  the connection 1-forms can be  summarized as
\be\label{w_i^j}
w^1_2,\ \ \  w^1_3=-\tau w^2\ \ \ \mbox{and}\ \ \ w^2_3=\tau w^1.
\ee
If one studies a surface $M$ isometrically immersed in $E(\kappa,\tau)$ one can firstly choose $\{f_1,f_2,f_3\}$ by rotating ${f_1,f_2}$ and leaving $f_3$ fixed at $p\in \mathcal{W}$ such that $f_1\in T_pM$,  and then choose the local orthonormal frame $\{e_1,e_2,e_3\}$, called the {\em adapted frame},  where
\be \label{framesurface}
e_1=f_1, \ \ \ e_2=\sin\beta f_2+\cos\beta f_3 \ \ \ \mbox{and}\ \ \  e_3=-\cos\beta f_2+\sin\beta f_3,
\ee so that $e_1$ and $e_2$ are tangent to the surface. This frame was firstly introduced for $\Bbb{S}^3$ in \cite{Gomes} and for $E(\kappa,\tau)$ in \cite{CuiGomes}.

\subsection{Structure equations} In this section we study the structure equations under the adapted frame in the open set $\mathcal{W}\subset M$ in $E(\kappa,\tau)$. The corresponding coframe $\{\theta^1,\theta^2,\theta^3\}$ is given by
\be \label{coframe}
\theta^1=w^1, \ \ \ \theta^2=\sin\beta w^2+\cos\beta w^3\ \ \ \mbox{and}\ \ \  \theta^3=-\cos\beta w^2+\sin\beta w^3.
\ee
The structure equations of $E(\kappa,\tau)$ are given by
\begin{equation*}
\bar{\Omega}^A_B=d\theta^A_B+\theta^A_C\w\theta^C_B\ \ \ \mbox{and}\ \ \ d\theta^A =-\theta^A_B\w\theta^B, \ \ \mbox{with}\ \ \theta^A_B+\theta^B_A=0.
\end{equation*}
Here and  from now on, we assume $1\leq A,B,\ldots\leq 3$,  $1\leq i,j,\ldots\leq 2$ and we will use the Einstein summation convention.
Since, when restricted to $M$, $\theta^3=0$ we get that $\theta^3_i\w\theta^i=0$. The Cartan's lemma thus implies $\theta^3_i=h_{ij}\theta^j$ with smooth functions $h_{ij}=h_{ji}$.
We now immediately deduce the structure equations of $M$
\begin{equation}\label{gauss11}
d\theta^i = - \theta^i_j\w\theta^j, \ \ \mbox{with}\ \  \theta^i_j+\theta^j_i=0\ \ \mbox{and}\ \
\Omega^i_j := d\theta^i_j+ \theta^i_k\w\theta^k_j=\bar{\Omega}^i_j-\theta^i_3\w\theta^3_j,
\end{equation}
where the last equation is called the {\it Gauss equation}. For $A=3$ and $B=i$ one gets the {\it Codazzi equations}:
\begin{equation} \label{codazzi11}
\bar{\Omega}^3_i=d\theta^3_i+ \theta^3_k\w\theta^k_i=d(h_{ij}\theta^j)+ h_{kj}\theta^j\w\theta^k_i=h_{ij|k}\theta^k\w\theta^j,
\end{equation}
where
\be \label{KKKK}
h_{ij|k}\theta^k:=dh_{ij}- h_{kj}\theta^k_i- h_{ik}\theta^k_j.
\ee
By using the classical notation $\bar{\Omega}^A_B={1\over2}\bar{R}^A_{BCD}\theta^C\w\theta^D\,$ and $\,\Omega^i_j={1\over2}R^i_{jkl}\theta^k\w\theta^l$
we get that the Gauss equation \eqref{gauss11} and Codazzi equation \eqref{codazzi11} are respectively given by
\be \label{Gauss12}
R^i_{jkl}=\bar{R}^i_{jkl}+h_{ik}h_{jl}-h_{ij}h_{kl},\nonumber
\ee
\be \label{Codazzi12}
h_{ij|k}-h_{ik|j}=-\bar{R}^3_{ijk}.\nonumber
\ee
%The Gaussian curvature $K$ and the mean curvature $H$ of $M$ are respectively given by
%\begin{equation*}
%K=R_{1212}:=R^1_{212} \ \ \ \mbox{and}\ \ \  2H=h_{11}+h_{22}.
%\end{equation*}

Throughout this paper, for simplicity of notation we denote $\beta_i=e_i(\beta)$, $\beta_{ij}=e_je_i(\beta)$ and $\beta_{ijk}=e_ke_je_i(\beta)$. The following lemma was actually proved in \cite{CuiGomes}, that gives us the  Levi-Civita connection on the surface $M$ and the second fundamental form  under the adapted frame. For the readers' convenience we give a proof here.

\begin{lem} (\cite{CuiGomes})\label{second}  Let $\nabla$ be the Levi-Civita connection on $M$. Then \[\nabla_{e_1}e_1=-(2H+\beta_2)\tan\beta e_2,\ \ \ \ \ \  \ \ \nabla_{e_2}e_2=-(2\tau+\beta_1)\tan\beta e_1,\] 
\be \nabla_{e_2}e_1=(\beta_1+2\tau)\tan\beta e_2\ \ \ \  \mbox{and}\ \  \ \ \nabla_{e_1}e_2=(2H+\beta_2)\tan\beta e_1,\label{NB}\ee
where $H:=(h_{11}+h_{22})/2$ is the mean curvature.
Moreover, the coefficients of the second fundamental form of $M$ are given by
\be \label{sec}
h_{11}=2H+\beta_2,\ \ \ h_{12}=h_{21}=-\tau-\beta_1 \ \ \ \mbox{and}\ \ \ h_{22}=-\beta_2.
\ee
Also, we have the relation 
\be \label{mean}
\cos\beta w^1_2(e_1)=2H+\beta_2.
\ee

\end{lem}

\begin{proof}
Since $\theta^3=0$ on $M$ we have
\[\cos\beta w^2=\sin\beta w^3,\]
which together with the expression \eqref{coframe} for $\theta^2$ gives us
\be \label{cof2}
\theta^1=w^1,\ \ \ \cos\beta\theta^2= w^3\ \ \ \mbox{and}\ \ \  \sin\beta\theta^2=w^2.
\ee
So, from equations \eqref{w_i^j} and \eqref{coframe} we have
\begin{eqnarray*}
d\theta^1 &=& dw^1=-w^1_2\w w^2-w^1_3\w w^3=-\sin\beta w^1_2\w \theta^2,\\
d\theta^2 &=& \sin\beta dw^2+\cos\beta dw^3 = \sin\beta(w_2^1-\tau\cos\beta\theta^2)\w \theta^1,\\
d\theta^3 &=& d\beta\w \theta^2-\cos\beta dw^2+\sin\beta dw^3 = d\beta\w \theta^2+\cos\beta w_1^2\w \theta^1+\tau(1+\sin^2\beta)\theta^1\w \theta^2.
\end{eqnarray*}
We thus obtain
\be  \theta^1_2=\sin\beta(w_2^1-\tau\cos\beta\theta^2),\ \  \ \  \cos\beta w^1_2(e_2)=-\beta_1 -\tau(1 + \sin^2\beta), \label{theta12}  \ee
and thus
 \[\nabla_{e_1}e_1=\theta^2_1(e_1)e_2=-\sin\beta w^1_2(e_1)e_2,\ \ \ \ \ \ \  \ \ \ \nabla_{e_2}e_2=\theta^1_2(e_2)e_1=-(2\tau+\beta_1)\tan\beta e_1,\] 
\be \nabla_{e_2}e_1=\theta^2_1(e_2)e_2=(2\tau+\beta_1)\tan\beta e_2\ \ \ \  \mbox{and}\ \  \ \ \nabla_{e_1}e_2=\theta^1_2(e_1)e_1=\sin\beta w^1_2(e_1) e_1.\label{NNNB}\ee

Now, using equations \eqref{w_i^j}, \eqref{framesurface} and \eqref{cof2} we compute
\begin{eqnarray*}
\bar{D}e_3 &=& \bar{D}(-\cos\beta f_2+\sin\beta f_3)\\
&=&\sin\beta d\beta f_2-\cos\beta(w_2^1f_1+w_2^3f_3)+\cos\beta d\beta f_3+\sin\beta(w_3^1f_1+w_3^2f_2)\\
&=&(-\cos\beta w_2^1-\tau\sin^2\beta\theta^2)e_1 + (d\beta+\tau\theta^1)e_2,
\end{eqnarray*}
where $\bar{D}$ is the Levi-Civita connection on $E(\kappa,\tau)$. On the other hand, $\bar{D}e_3=\theta_3^1e_1+\theta^2_3e_2$, whence
\[ \label{connectoin1}
\theta^1_3=-\cos\beta w^1_2-\tau \sin^2\beta \theta^2\ \ \ \mbox{and}\ \ \ \theta^2_3=d\beta+\tau\theta^1.
\]
Hence, from the second equation of \eqref{theta12} and  $\theta^3_i=h_{ij}\theta^j$, we get
\be \label{secd}
h_{11}=\cos\beta w^1_2(e_1),\ \ \  h_{12}=\cos\beta w^1_2(e_2)+\tau\sin^2\beta = -\tau-\beta_1\ \ \mbox{and}\ \ h_{22} =-\beta_2. 
\ee
Since by the definition $2H=h_{11}+h_{22}$, we get (\ref{mean}) and then (\ref{NB}) and (\ref{sec}) follow from  (\ref{mean}), (\ref{NNNB}) and  (\ref{secd}) immediately.
\end{proof}

By using the formula of the Riemannian curvature tensor in \cite{Da}, one can obtain the Codazzi equations under the adapted frame. See the proof in \cite{CuiGomes}.
\begin{lem} (\cite{CuiGomes}) The Codazzi equations under the adapted frame are given by
\be\beta_1w^1_2(e_1)\cos\beta=\beta_2(2\tau+\beta_1)\label{Codazzi1}\ee and
\begin{eqnarray}\label{Codazzi2}
\nonumber0&=&\cos\beta e_2(w^1_2(e_1))+\sin\beta \cos\beta (w^1_2(e_1))^2+\beta_{11}+2(\tau+\beta_1)(2\tau+\beta_1)\tan\beta\\
&&+(\kappa-4\tau^2)\sin\beta \cos\beta.
\end{eqnarray}
\end{lem}

At the end of this section, we shall derive some identities which will be useful in the next section. 
From (\ref{mean}) and (\ref{Codazzi1}), we have  $\beta_1(2H+\beta_2)=\beta_2(2\tau+\beta_1)$, that is\be \tau\beta_2=H\beta_1.\label{beta12}\ee
%This identity enables us to express the related derivatives of $\beta$ to $\beta_1$, $\beta_{11}$ and $\beta_{111}$. 
According to Lemma \ref{second}, we observe  that \beq \beta_{12}&=&e_2e_1\beta=[e_2,e_1]\beta+\beta_{21}=\nabla_{e_2}e_1\beta-\nabla_{e_1}e_2\beta+\beta_{21}\nonumber\\&=&(2\tau+\beta_1)\tan\beta\beta_2-(2H+\beta_2)\tan\beta\beta_1+\beta_{21}=\beta_{21}.\nonumber\label{beta111A}\eeq
By (\ref{beta12}), if $H$ is a constant, we have $\tau\beta_{21}=H\beta_{11}$ and $\beta_1\beta_{21}=\beta_2\beta_{11}$. Then a similar way shows that \[ \beta_{112}=e_2e_1(\beta_1)=[e_2,e_1]\beta_1+\beta_{121}=\nabla_{e_2}e_1\beta_1-\nabla_{e_1}e_2\beta_1+\beta_{121}=\beta_{121}.\label{beta112}\]
Thus by (\ref{beta12}) again, for the later use, we deduce that  \be \beta_{12}=\beta_{21}=(H/ \tau)\beta_{11},\ \ \beta_{22}=(H/ \tau)\beta_{12}=(H/ \tau)\beta_{21}=(H/ \tau)^2\beta_{11}\label{BBBd}\ee
and \be \beta_{122}=(H/ \tau)\beta_{112}=(H/ \tau)\beta_{121}=(H/ \tau)^2\beta_{111}.\label{beta113}\ee

\section{Simons' type formula for cmc surafces}

In this section we consider the tensor $\Phi:=A-HI$ with constant $H$. By Lemma \ref{second} and (\ref{beta12}), we have \beq |\Phi|^2&=&|A|^2-2H^2=(2H+\beta_2)^2+2(\tau+\beta_1)^2+\beta_2^2-2H^2\nonumber\\
&=&2(\tau+\beta_1)^2+2(H+\beta_2)^2=2[1+({H/\tau})^2](\tau+\beta_1)^2.\label{Phi2}\eeq

 We now start to compute the Laplacian of $|\Phi|^2$. By using (\ref{beta12}),  (\ref{BBBd}) and  (\ref{beta113}), we compute that
\beq \Delta |\Phi|^2&=&e_1e_1(|\Phi|^2)+e_2e_2(|\Phi|^2)-\nabla_{e_1}e_1 |\Phi|^2-\nabla_{e_2}e_2 |\Phi|^2\nonumber\\
&=&4[1+({H/\tau})^2][\beta_{11}^2+(\tau+\beta_1)\beta_{111}+\beta_{12}^2+(\tau+\beta_1)\beta_{122}\nonumber\\
&&+(2H+\beta_2)(\tau+\beta_1)\tan\beta\beta_{12}+(2\tau+\beta_1)(\tau+\beta_1)\tan\beta\beta_{11}]\nonumber\\
&=&4[1+({H/\tau})^2]^2[\beta_{11}^2+(\tau+\beta_1)\beta_{111}+(2\tau+\beta_1)(\tau+\beta_1)\tan\beta\beta_{11}].\label{DeltaPhi2}
\eeq

Next, we shall express $\beta_{11}$ and $\beta_{111}$ to the expressions involving  in $\beta_1$.
First, by Lemma \ref{second} we have $2H+\beta_2=\cos\beta w^1_2(e_1)$ and thus we get
\begin{equation*}\beta_{22}=-\tan\beta\beta_2(2H+\beta_2)+\cos\beta e_2(w^1_2(e_1)).
\end{equation*}
So, the Codazzi equation \eqref{Codazzi2} becomes
\beq
\beta_{11}+\beta_{22} &=& -\tan\beta[(2H+\beta_2)\beta_2+(2H+\beta_2)^2+2(\tau+\beta_1)(2\tau+\beta_1)]\nonumber\\
&&-(\kappa-4\tau^2)\sin\beta \cos\beta.\label{c2-H}
\eeq
Again, by using  (\ref{beta12}) and  (\ref{beta113}), the equation (\ref{c2-H}) becomes \be [1+({H/\tau})^2] \beta_{11}=-2[1+({H/\tau})^2](\tau+\beta_1)(2\tau+\beta_1)\tan\beta-(\kappa-4\tau^2)\sin\beta\cos\beta,\label{beta11}\ee
and  taking the derivative by $e_1$ gives us \beq [1+({H/\tau})^2] \beta_{111}&=&-2[1+({H/\tau})^2][(3\tau+2\beta_1)\beta_{11}\tan\beta+(\tau+\beta_1)(2\tau+\beta_1)\sec^2\beta\beta_1]\nonumber\\
&&-(\kappa-4\tau^2)\cos(2\beta)\beta_1\nonumber\\
&=&4[1+({H/\tau})^2](3\tau+2\beta_1)(\tau+\beta_1)(2\tau+\beta_1)\tan^2\beta+2(\kappa-4\tau^2)(3\tau+2\beta_1)\sin^2\beta\nonumber\\
&&-2[1+({H/\tau})^2](\tau+\beta_1)(2\tau+\beta_1)\sec^2\beta\beta_1-(\kappa-4\tau^2)\cos(2\beta)\beta_1.\label{beta111}\eeq
Substituting (\ref{beta11}) and (\ref{beta111}) into (\ref{DeltaPhi2}), by a straightforward computation we have
\beq \Delta |\Phi|^2&=&4[1+({H/\tau})^2]^2\Big\{\beta_{11}^2+2(\tau+\beta_1)^2(2\tau+\beta_1)[2(2\tau+\beta_1)\tan^2\beta-\beta_1]\Big\}\nonumber\\
&&+4[1+({H/\tau})^2](\kappa-4\tau^2)(\tau+\beta_1)[4(\tau+\beta_1)\sin^2\beta-\beta_1\cos^2\beta]. \label{DeltaPhi2F}\eeq

\begin{lem}\label{AAA} \be |\nabla A|^2=2[1+({H/\tau})^2]^2[\beta_{11}^2+4(\tau+\beta_1)^2(2\tau+\beta_1)^2\tan^2\beta].\label{Asquare}\ee
\end{lem}
\begin{proof} By (\ref{KKKK}), (\ref{BBBd}) and Lemma \ref{second},  we compute that
\beq h_{11|1}&=&e_1(2H+\beta_2)-h_{21}\theta^2_1(e_1)-h_{12}\theta^2_1(e_1)\nonumber\\
&=&\beta_{21}-2(\tau+\beta_1)(2H+\beta_2)\tan\beta\nonumber\\
&=&(H/\tau)[\beta_{11}-2(\tau+\beta_1)(2\tau+\beta_1)\tan\beta].\nonumber
\eeq
A similar computation yields
\[h_{11|2}=(H/\tau)^2\beta_{11}+2(\tau+\beta_1)(2\tau+\beta_1)\tan\beta,\ \ \ \  h_{12|1}=-\beta_{11}-2(H/\tau)^2(\tau+\beta_1)(2\tau+\beta_1)\tan\beta\]
and the following identities hold true: \[h_{12|2}=-h_{11|1}=-h_{22|1}\  \ \ \mbox{and}\  \ \ h_{12|1}=-h_{22|2}.\]
Thus by definition,  \[|\nabla A|^2=h_{11|1}^2+h_{11|2}^2+h_{22|1}^2+h_{22|2}^2+2(h_{12|1}^2+h_{12|2}^2)=2(2h_{11|1}^2+h_{11|2}^2+h_{12|1}^2),\]
and Eq. (\ref{Asquare}) follows by a straightforward computation.
\end{proof} 

By (\ref{Phi2}), (\ref{DeltaPhi2F}), (\ref{Asquare}) and $|\nabla\Phi|^2=|\nabla A|^2$, we have
\beq {1\over 2}\Delta |\Phi|^2
&=&|\nabla\Phi|^2-|\Phi|^2[|\Phi|^2-2(H^2+\tau^2)]\nonumber\\
&&+(\kappa-4\tau^2)\Big\{4|\Phi|^2\sin^2\beta-2[1+({H/\tau})^2](2\tau+\beta_1)\beta_1\cos^2\beta\Big\}\nonumber\\
&&+2(\kappa-4\tau^2)[1+({H/\tau})^2]\tau\beta_1\cos^2\beta\nonumber\\
&=&|\nabla\Phi|^2-|\Phi|^2[|\Phi|^2-2(H^2+\tau^2)]\nonumber\\
&&+(\kappa-4\tau^2)[4|\Phi|^2\sin^2\beta-(|\Phi|^2-2(H^2+\tau^2))\cos^2\beta]\nonumber\\
&&+2(\kappa-4\tau^2)[1+({H/\tau})^2]\tau\beta_1\cos^2\beta. \label{DeltaPhi2FF}
\eeq
 We now deal with the last term in (\ref{DeltaPhi2FF}). Inspired by the computations in   \cite{Hu} and \cite{TU}, we prove the following lemma.

\begin{lem}\label{LLL}  Let $T:=f_3-\la f_3,N\ra N$ be the projection of the vertical vector field $f_3=\xi$ on $\mathcal{W}\subset M$. Then 
 \be {1\over2}\Delta|T|^2-div(\nabla_TT)=2\tau(\beta_1\cos^2\beta+2\tau\sin^2\beta). \label{CVC}\ee\end{lem}

%We deduce from (\ref{c2-H}) that
%\beq  \Delta\beta&=&\beta_{11}+\beta_{22}+[(2H+\beta_2)\beta_2+(2\tau+\beta_1)\beta_1]\tan\beta.\label{lap}  \eeq

\begin{proof} By (\ref{framesurface}), we can express $T=f_3-\sin\beta e_3=\cos\beta e_2$. Then we compute 
\[\nabla_TT=\cos\beta\nabla_{e_2}(\cos\beta e_2)=-\sin\beta\cos\beta [\beta_2e_2+(2\tau+\beta_1)e_1]=-(1/2)\sin(2\beta)(\nabla\beta+2\tau e_1).\]
Noticing that \[ {1\over2}\Delta|T|^2={1\over2}\Delta(\cos^2\beta)=-\cos(2\beta)|\nabla\beta|^2-\sin\beta\cos\beta\Delta\beta,\] we get
 \beq div(\nabla_TT)&=&\la \nabla_{e_i}\nabla_TT,e_i\ra=-\cos(2\beta)(|\nabla\beta|^2+2\tau \beta_1)-\sin\beta\cos\beta(\Delta\beta+2\tau\la \nabla_{e_i}e_1,e_i\ra)\nonumber\\
&=&{1\over2}\Delta|T|^2-2\tau(\beta_1\cos^2\beta+2\tau\sin^2\beta),\nonumber\eeq
which proves (\ref{CVC}).
\end{proof}

By (\ref{DeltaPhi2FF}), (\ref{CVC}) and noticing that the contact function is $C=\sin\beta$, we have 
\beq &&{1\over 2}\Delta |\Phi|^2-(\kappa-4\tau^2)[1+({H/\tau})^2]\Big({1\over2}\Delta|T|^2-div(\nabla_TT)\Big)\nonumber\\
&=&|\nabla\Phi|^2-|\Phi|^2[|\Phi|^2-2(H^2+\tau^2)]\nonumber\\
&&+(\kappa-4\tau^2)[4|\Phi|^2\sin^2\beta-(|\Phi|^2-2(H^2+\tau^2))\cos^2\beta]\nonumber\\
&&-4(\kappa-4\tau^2)(H^2+\tau^2)\sin^2\beta\nonumber\\
&=&|\nabla\Phi|^2-|\Phi|^2[|\Phi|^2-2(H^2+\tau^2)]\nonumber\\
&&+(\kappa-4\tau^2)[|\Phi|^2(5C^2-1)-2(H^2+\tau^2)(3C^2-1)].\label{DP}\eeq

Although the above formula is derived on $\mathcal{W}\subset M$, it can be extended to the whole $M$ since the complimentary  $\mathcal{W}^C$ has empty interior.

\begin{rem} The Berger sphere is the unit sphere $\Bbb{S}^3$ endowed with the metric
\[ \la X,Y\ra=\frac{4}{\kappa}\Big[\la X,Y\ra_{\Bbb{S}^3}+\Big(\frac{4\tau^2}{\kappa}-1\Big)\la X,\xi \ra_{\Bbb{S}^3} \la Y,\xi \ra_{\Bbb{S}^3}\Big],\label{metric1}\]
where $\langle\, , \,\rangle_{\Bbb{S}^3}$ stands for the standard metric on $\Bbb{S}^{3}$, $\kappa>0$ and $\tau\neq0$. Although in this paper we assume $\kappa\neq4\tau^2$, the formula (\ref{DP}) still holds for $\kappa=4\tau^2$. By taking $\kappa=4\tau^2=4c>0$, the sphere is the Riemannian space form $\Bbb{S}^3(c)$ with constant sectional curvature $c$. In this case,  (\ref{DP}) reduces to 
\[{1\over 2}\Delta |\Phi|^2=|\nabla\Phi|^2-|\Phi|^2[|\Phi|^2-2(H^2+c)],\]
which is the well known formula in \cite{NoS} by observing that $tr\Phi^3=0$ for surfaces.
\end{rem}

{\bf Proof of Theorem \ref{mainthm}}:
The Simons' type formula (\ref{Simon}) is a direct consequence by integrating  (\ref{DP})  on $M$, no matter whether $M$ is orientable or not. The equality holds if and only if $|\nabla A|=|\nabla \Phi|=0$, and this implies that $\sin\beta=0$ from (\ref{beta11}) and (\ref{Asquare}). Therefore, the vertical Killing vector field of the fibration  $\Pi:E(\kappa,\tau)\rightarrow\Bbb{M}^2(\kappa)$ is tangent to $M$,  so $\Pi(M)$ is a closed curve $\gamma$ in $S^2$ and then $M=\Pi^{-1}(\gamma)$ is the Hopf cylinder in $E(\kappa,\tau)$. In this case, $|A|^2=2(2H^2+\tau^2)$ follows from (\ref{Phi2}) immediately. \qed
 
\begin{rem} The authors in \cite{BeDi} proved that the Hopf cylinder is the only surface in $E(\kappa,\tau)$ ($\tau\neq 0$) with parallel second fundamental form. This fact can be used to prove the case of the equality in Theorem \ref{mainthm} instead of using Eqs.  (\ref{beta11}) and (\ref{Asquare}).
\end{rem}

As an application, we shall use Theorem \ref{mainthm}  to give a pinching result of the second fundamental form.  Observe that in the case of $\kappa>4\tau^2$, the equation 
\[x^2-[2(H^2+\tau^2)+(\kappa-4\tau^2)(5C^2-1)]x+2(\kappa-4\tau^2)(H^2+\tau^2)(3C^2-1)=0\]
has two distinct real solutions: 
\[x_1=a(\kappa,\tau,H,C):=[2(H^2+\tau^2)+(\kappa-4\tau^2)(5C^2-1)-\sqrt{\rho}]/2\]
and \[x_2=b(\kappa,\tau,H,C):=[2(H^2+\tau^2)+(\kappa-4\tau^2)(5C^2-1)+\sqrt{\rho}]/2,\]
where \[\rho=4(H^2+\tau^2)^2+4(H^2+\tau^2)(\kappa-4\tau^2)(1-C^2)+(\kappa-4\tau^2)^2(5C^2-1)^2.\]
Obviously, it holds the relation  $a(\kappa,\tau,H,C)\leq H^2+\tau^2<2(H^2+\tau^2)\leq b(\kappa,\tau,H,C) $. By Theorem \ref{mainthm} and the identity $|\Phi|^2=|A|^2-2H^2$, we get immediately

\begin{cor}\label{maincor}
Let $f:M\rightarrow E(\kappa,\tau)$ ($\kappa>4\tau^2$ and $\tau\neq 0$) be an immersion of a compact cmc $H$-surface. If  $a(\kappa,\tau,H,C)+2H^2\leq |A|^2\leq b(\kappa,\tau,H,C)+2H^2$,  then $M$ is a Hopf cylinder with $|A|^2=2(2H^2+\tau^2)$. In particular,  if  $3H^2+\tau^2\leq |A|^2\leq 2(2H^2+\tau^2)$,  then $M$ is a Hopf cylinder with $|A|^2=2(2H^2+\tau^2)$. 
\end{cor}

\begin{rem} When $H=0$, in the case of $\kappa>4\tau^2$, the authors in \cite{Hu} obtained that $M$ is a Clifford torus in the Berger sphere under the assumption that $a(\kappa,\tau,0,C)\leq |A|^2\leq b(\kappa,\tau,0,C)$,  by observing the fact that the Clifford torus is the only minimal Hopf torus in the Berger sphere (see \cite{BaFe} or \cite{Pink}). We also mention that there exist cmc (even minimal) tori which are not Hopf tori in the Berger sphere (see \cite{To}).
\end{rem}

\begin{rem} Similar to the Remark 4.8 in \cite{Hu}, the case of $\kappa<4\tau^2$ for Corollary \ref{maincor} is still open.
\end{rem}

\end{document}